\documentclass[11pt]{article}

\usepackage[a4paper,margin=1.15in]{geometry}
\usepackage{amsmath,amssymb,amsthm,mathtools}
\usepackage{mathrsfs}
\usepackage{bm}
\usepackage{microtype}
\usepackage{hyperref}

\hypersetup{
    colorlinks=true,
    linkcolor=blue,
    citecolor=blue,
    urlcolor=blue
}

\newcommand{\R}{\mathbb{R}}
\newcommand{\Sph}{\mathbb{S}}
\newcommand{\E}{\mathbb{E}}
\newcommand{\Pp}{\mathbb{P}}

\newcommand{\Tr}{\operatorname{Tr}}
\newcommand{\Id}{\operatorname{Id}}
\newcommand{\Cov}{\operatorname{Cov}}
\newcommand{\codim}{\operatorname{codim}}
\newcommand{\ip}[2]{\left\langle #1,#2\right\rangle}
\newcommand{\norm}[1]{\left\|#1\right\|}

\theoremstyle{plain}
\newtheorem{theorem}{Theorem}[section]
\newtheorem{proposition}[theorem]{Proposition}

\theoremstyle{remark}
\newtheorem{remark}[theorem]{Remark}

\title{\bf Geometric Bounds for the Mean Gauge\\
and the Mean Width in Isotropic Position}

\author{
  Silouanos Brazitikos\thanks{University of Crete, Greece. Email: \texttt{silouanb@uoc.gr}}
  \and
Christos Pandis \thanks{University of Crete, Greece. Email: \texttt{chrpandis@gmail.com}
}}
\date{}

\begin{document}
\maketitle

\begin{abstract}
Let $K\subset\R^n$ be an origin-symmetric convex body and assume that
its uniform probability measure is isotropic in the probabilistic
normalization, namely
\[
  \int_K x\otimes x\,d\mu_K(x)=\Id_n.
\]
We give deterministic geometric proofs of
\[
 M(K)\leq C\frac{\log(n)}{\sqrt n}
 \qquad\text{and}\qquad
 M^*(K)\leq C\sqrt n\,\log(n),
\]
where
\[
 M(K)=\int_{\Sph^{n-1}}\norm{\theta}_K\,d\sigma(\theta),
 \qquad
 M^*(K)=\int_{\Sph^{n-1}}h_K(\theta)\,d\sigma(\theta).
\]
Combining both estimates yield
\[
M(K)M^{\ast}(K)\leq  C\log^2n.
\]

The first proof uses a quadratic aggregate of dyadic centroid bodies.
The second uses the analogous weighted aggregate of the Laplace bodies
$p\{\Lambda_K\leq p\}^{\circ}$, which are equivalent to the centroid
bodies by the work of Klartag and E.~Milman. In both cases, curvature at
each dyadic scale outside a subspace of codimension $O(p)$ leads, via
the min--max principle, Legendre duality, and the spherical Laplacian,
to the required estimate. The only high-dimensional input is the
dimension-free small-ball consequence of the slicing theorem.
\end{abstract}

\section{Introduction}

Let \(K\subset\R^n\) be an origin-symmetric convex body. Its Minkowski
functional and support function are defined by
\[
    \norm{x}_K
    =
    \inf\{t>0:x\in tK\},
    \qquad
    h_K(u)
    =
    \sup_{x\in K}\ip{x}{u}.
\]
We consider the mean gauge and the mean width, in the normalization
\[
    M(K)
    =
    \int_{\Sph^{n-1}}\norm{\theta}_K\,d\sigma(\theta),
    \qquad
    M^*(K)
    =
    \int_{\Sph^{n-1}}h_K(\theta)\,d\sigma(\theta),
\]
where \(\sigma\) is the Haar probability measure on
\(\Sph^{n-1}\). Since \(K=-K\),
\[
    M^*(K)=M(K^\circ).
\]
Equivalently, if \(G\) is a standard Gaussian vector in \(\R^n\),
then polar integration gives
\[
    \E\norm{G}_K
    \asymp
    \sqrt n\,M(K),
    \qquad
    \E h_K(G)
    \asymp
    \sqrt n\,M^*(K).
\]

Throughout the paper, \(\mu_K\) denotes the uniform probability
measure on \(K\). We say that \(K\) is in probabilistic isotropic
position if
\[
    \int_K x\,d\mu_K(x)=0,
    \qquad
    \Cov(\mu_K)=\Id_n.
\]
More generally, if \(\nu\) is a log-concave probability measure on
\(\R^r\), with density \(f_\nu\) and positive definite covariance
matrix, its isotropic constant is
\[
    L_\nu
    =
    \norm{f_\nu}_\infty^{1/r}
    \bigl(\det\Cov(\nu)\bigr)^{1/(2r)}.
\]
This quantity is invariant under invertible affine transformations.
In particular, if \(\nu\) is isotropic, then
\[
    L_\nu=\norm{f_\nu}_\infty^{1/r}.
\]
For a convex body \(K\), we write \(L_K=L_{\mu_K}\). Thus, if \(K\)
has volume one and is in the traditional isotropic position,
\[
    \Cov(\mu_K)=L_K^2\Id_n,
\]
whereas, if \(K\) is in probabilistic isotropic position, then
\[
    L_K=|K|^{-1/n}.
\]
We refer to \cite{BGVV} for the standard background on isotropic
position, isotropic constants, centroid bodies, and the equivalent
formulations of the slicing problem.

Bourgain's slicing problem asks whether there exists an absolute
constant \(C>0\) such that
\[
    L_\nu\leq C
\]
for every log-concave probability measure \(\nu\), independently of
its dimension. Equivalently, if
\[
    L_n
    =
    \sup\left\{
        L_\nu:
        \nu\text{ is a log-concave probability measure on }\R^r,
        \ 1\leq r\leq n
    \right\},
\]
then the conjecture asserts that
\[
    \sup_{n\geq1}L_n<\infty.
\]
The problem originates in Bourgain's work
\cite{Bourgain1986}. Guan \cite{Guan} obtained the estimate
\[
    L_n\leq C\log\log(n),
\]
and the conjecture was subsequently resolved in the affirmative by
Klartag and Lehec \cite{KlartagLehec}, using Guan's bound.

The slicing theorem enters both arguments of the present paper only
through a dimension-free small-ball estimate. Indeed, if \(\nu\) is
an isotropic log-concave probability measure on \(\R^r\), then, for
every \(z\in\R^r\) and every \(\varepsilon>0\),
\begin{equation}\label{eq:small-ball}
    \nu\bigl(z+\varepsilon\sqrt r\,B_2^r\bigr)
    \leq
    (C\varepsilon)^r.
\end{equation}
More precisely, without invoking slicing, one always has
\begin{equation}\label{eq:small-ball-L}
    \nu\bigl(z+\varepsilon\sqrt r\,B_2^r\bigr)
    \leq
    (C\varepsilon L_\nu)^r.
\end{equation}
This follows directly from the density bound. 
Thus \eqref{eq:small-ball} is precisely the consequence of
\eqref{eq:small-ball-L} obtained from the slicing theorem.

The estimate \eqref{eq:small-ball-L} should be distinguished from
the deeper small-ball theory for Euclidean norms of log-concave
random vectors. Paouris \cite{PaourisSmallBall} proved that, for an
isotropic log-concave random vector \(X\) in \(\R^r\),
\[
    \Pp\bigl(|X|\leq\varepsilon\sqrt r\bigr)
    \leq
    (C\varepsilon)^{c\sqrt r}.
\]
This estimate is independent of the isotropic constant, but its
exponent would yield only a codimension bound of order \(p^2\) in the
spectral arguments below, whereas our proofs require codimension
\(O(p)\). More recently, Bizeul
\cite{BizeulSmallBall} obtained a sharp estimate with exponent
proportional to the dimension, uniformly over the center of the ball,
and used it to give an alternative proof of the slicing theorem.
That result can therefore replace \eqref{eq:small-ball} in the
arguments below, but it is not a weaker independent substitute for
the slicing input.

We first study the mean gauge. Its estimation in isotropic position
is closely connected with the \(MM^*\)-theory and with the comparison
of isotropic log-concave measures with Gaussian measure. Earlier
estimates based on \(L_q\)-centroid bodies were obtained in
\cite{GiannopoulosMilman,GSTV,Vritsiou}. More recently, Bizeul and
Klartag \cite{BizeulKlartag} obtained an estimate involving the
third-moment parameter \(\kappa_n\). Letwin's quadratic-form estimate
\cite{LetwinKLS} yields a dimension-free bound for this parameter,
and consequently their result gives the optimal bound
\[
    M(K)
    \leq
    C\sqrt{\frac{\log(n)}{n}}
\]
for every convex body in probabilistic isotropic position.

The first argument of this paper gives the slightly weaker estimate
\begin{equation}\label{eq:intro-M}
    M(K)
    \leq
    C\frac{\log(n)}{\sqrt n}.
\end{equation}
The purpose is not to improve the best known bound, but to provide a
direct geometric proof. We introduce a convex body \(R_K\) by
quadratically combining the dyadic centroid bodies \(Z_p(K)\).
A weighted-covariance argument shows that the Hessian of the
\(p\)-th summand has curvature of order \(p\) outside a subspace of
codimension \(O(p)\). A dyadic eigenvalue-counting argument then gives
a logarithmic inverse-trace estimate. Legendre duality and integration
of the resulting Laplacian bound over the sphere complete the proof.

We next consider the dual quantity \(M^*(K)\). In the volume-one
isotropic normalization, E.~Milman \cite{MilmanMeanWidth} proved
\[
    M^*(K)
    \leq
    C\sqrt n\,(\log(n))^2L_K.
\]
After the resolution of the slicing problem, this yields
\[
    M^*(K)
    \leq
    C\sqrt n\,(\log(n))^2.
\]
Bizeul \cite{BizeulMMstar} recently obtained the optimal estimate
\[
    M^*(K)
    \leq
    C\sqrt{n\log(n)}.
\]
We prove the intermediate bound
\begin{equation}\label{eq:intro-Mstar}
    M^*(K)
    \leq
    C\sqrt n\,\log(n).
\end{equation}
As in the mean-gauge argument, the proof is deterministic once the
slicing theorem is admitted: it uses no stochastic localization,
martingale inequality, heat flow, and moreover, it avoids the Milman--Pisier theorem that E. Milman used.

Combining \eqref{eq:intro-M} and \eqref{eq:intro-Mstar}, we obtain
\[
    M(K)M^*(K)
    \leq
    C(\log(n))^2.
\]
Thus the resulting \(MM^*\)-estimate loses only one logarithmic factor
from the optimal order \(O(\log(n))\). We also mention the recent
work of Paouris and Pathak \cite{PaourisPathak}, who obtained optimal
affine mean-width and metric-entropy estimates for general convex
bodies.

The auxiliary body used in the proof of \eqref{eq:intro-Mstar} is
naturally suggested by the work of Klartag and E.~Milman. Define the
logarithmic Laplace transform of \(\mu_K\) and its Legendre transform
by
\[
    \Lambda_K(\xi)
    =
    \log\int_K e^{\ip{\xi}{x}}\,d\mu_K(x),
    \qquad
    \Lambda_K^*(x)
    =
    \sup_{\xi\in\R^n}
    \bigl\{
        \ip{\xi}{x}-\Lambda_K(\xi)
    \bigr\}.
\]
Klartag and E.~Milman \cite{KlartagMilmanLaplace} showed that the
level sets of \(\Lambda_K\) and the centroid bodies encode the same
geometry. More precisely, for \(2\leq p\leq n\),
\[
    \{\Lambda_K\leq p\}
    \simeq
    pZ_p(K)^\circ.
\]
Equivalently, the Laplace body
\[
    L_p(K)
    :=
    p\{\Lambda_K\leq p\}^{\circ}
\]
satisfies
\[
    L_p(K)\simeq Z_p(K).
\]
This equivalence is one of the central observations of
\cite{KlartagMilmanLaplace} and explains why the Laplace bodies are
the natural objects for the dual problem. They retain the scale and
containment properties of the centroid bodies, while their boundary
curvature is expressed directly through
\[
    \nabla^2\Lambda_K(\xi)
    =
    \Cov(\mu_{K,\xi}),
\]
where
\[
    d\mu_{K,\xi}(x)
    =
    e^{\ip{\xi}{x}-\Lambda_K(\xi)}\,d\mu_K(x)
\]
is the exponential tilt of \(\mu_K\). The required curvature estimate
is therefore reduced to a spectral estimate for tilted covariance
matrices. The same small-ball input \eqref{eq:small-ball} shows that,
if \(p=\Lambda_K(\xi)\geq2\), then
\[
    \#\left\{
        i:
        \lambda_i\bigl(\Cov(\mu_{K,\xi})\bigr)>D
    \right\}
    \leq
    2p
\]
for an absolute constant \(D>0\). Legendre duality turns this upper
spectral estimate into lower curvature for the gauges of the Laplace
bodies, and a dyadic aggregation followed by the same inverse-trace
and spherical-Laplacian mechanism yields
\eqref{eq:intro-Mstar}.

We finally record explicitly what the two arguments give if the
slicing theorem is not used. If \(K\) is in probabilistic isotropic
position, then \eqref{eq:small-ball-L} and the general bound
\(L_\nu\leq L_n\) give
\[
    M(K)
    \leq
    C L_n\frac{\log(n)}{\sqrt n}
\]
and
\[
    M^*(K)
    \leq
    C L_n\sqrt n\,\log(n).
\]
More precisely, the factor \(L_n\) may be replaced in the first
estimate by the supremum of the isotropic constants of the marginals
appearing in the weighted-covariance argument, and in the second by
the corresponding supremum over the whitened marginals of exponential
tilts.

Equivalently, suppose that \(K\) has volume one and is in the
traditional isotropic position,
\[
    \Cov(\mu_K)=L_K^2\Id_n.
\]
Applying the preceding estimates to
\[
    \widetilde K=L_K^{-1}K,
\]
and using
\[
    M(\widetilde K)=L_KM(K),
    \qquad
    M^*(\widetilde K)=\frac{1}{L_K}M^*(K),
\]
we obtain
\[
    M(K)
    \leq
    C\frac{L_n}{L_K}
    \frac{\log(n)}{\sqrt n}
\]
and
\[
    M^*(K)
    \leq
    C L_KL_n\sqrt n\,\log(n).
\]
Thus, in both arguments, the slicing theorem is used only to replace
the isotropic constants of the lower-dimensional measures arising in
the respective spectral constructions by an absolute constant.

\section{The mean-gauge estimate}

\begin{theorem}\label{thm:M}
Let $K=-K\subset\R^n$, and assume that the uniform probability measure
$\mu_K$ satisfies $\Cov(\mu_K)=\Id_n$. Then
\[
 M(K)\leq C\frac{\log(n)}{\sqrt n}.
\]
\end{theorem}

\begin{proof}
We may assume that $n\geq4$. Let $X\sim\mu_K$. For $p\geq2$, define
\[
 h_{Z_p(K)}(u)=\bigl(\E|\ip{X}{u}|^p\bigr)^{1/p},
\]
and define a symmetric convex body $R=R_K$ by
\begin{equation}\label{eq:R-def}
 h_R(u)^2
 =
 |u|^2+
 \sum_{\substack{p=2^j\\4\leq p\leq n}}
 h_{Z_p(K)}(u)^2.
\end{equation}
Since $X\in K$ almost surely,
\[
 h_{Z_p(K)}(u)\leq h_K(u).
\]
Moreover, if $Y=\ip{X}{u}$, the one-dimensional reverse moment
inequality gives
\[
 h_K(u)\geq\E|Y|\geq c(\E Y^2)^{1/2}=c|u|.
\]
There are $O\bigl(\log(n)\bigr)$ dyadic scales, and therefore
\[
 h_R(u)^2\leq C\log(n)\,h_K(u)^2.
\]
Equivalently,
\[
 R\subseteq C\sqrt{\log(n)}\,K,
 \qquad
 \norm{x}_K\leq C\sqrt{\log(n)}\,\norm{x}_R.
\]
Thus
\begin{equation}\label{eq:M-reduction}
 M(K)
 \leq
 C\sqrt{\log(n)}
 \left(
 \int_{\Sph^{n-1}}\norm{\theta}_R^2\,d\sigma(\theta)
 \right)^{1/2}.
\end{equation}

Set
\[
 \Phi(u)=\frac12h_R(u)^2,
 \qquad
 \Psi=\Phi^*=\frac12\norm{\cdot}_R^2.
\]
The function $\Psi$ is $2$-homogeneous. The polar-coordinate formula
for the Euclidean Laplacian, followed by integration of the spherical
Laplacian, gives
\begin{equation}\label{eq:spherical-M}
 n\int_{\Sph^{n-1}}\norm{\theta}_R^2\,d\sigma(\theta)
 =
 \int_{\Sph^{n-1}}\Delta\Psi(\theta)\,d\sigma(\theta).
\end{equation}
Since the Euclidean term in \eqref{eq:R-def} makes $\Phi$ uniformly
strongly convex, Legendre duality gives
\[
 \nabla^2\Psi(x)
 =
 \left[\nabla^2\Phi\bigl(\nabla\Psi(x)\bigr)\right]^{-1}.
\]
It is therefore enough to show that
\begin{equation}\label{eq:M-inverse-trace-goal}
 \Tr\left[\nabla^2\Phi(u)\right]^{-1}
 \leq
 C\log(n),
 \qquad u\neq0.
\end{equation}

For a dyadic $p\in[4,n]$, write
\[
 \Phi_p(u)=\frac12\bigl(\E|\ip{X}{u}|^p\bigr)^{2/p},
 \qquad
 m=\bigl(\E|\ip{X}{u}|^p\bigr)^{1/p}.
\]
For $v\in\R^n$, set
\[
 Q_{p,u}(v)
 =
 m^{2-p}\E\bigl[|\ip{X}{u}|^{p-2}\ip{X}{v}^2\bigr]
\]
and
\[
 \ell_{p,u}(v)
 =
 m^{1-p}\E\bigl[|\ip{X}{u}|^{p-2}\ip{X}{u}\ip{X}{v}\bigr].
\]
Direct differentiation gives
\begin{equation}\label{eq:M-Hessian}
 D^2\Phi_p(u)[v,v]
 =
 (p-1)Q_{p,u}(v)-(p-2)\ell_{p,u}(v)^2.
\end{equation}
Cauchy--Schwarz yields $\ell_{p,u}(v)^2\leq Q_{p,u}(v)$, so
$D^2\Phi_p(u)$ is positive semidefinite.

We claim that there is a subspace $G_{p,u}\subseteq\R^n$ such that
\begin{equation}\label{eq:M-good-space}
 \codim G_{p,u}\leq Cp,
 \qquad
 Q_{p,u}(v)\geq b|v|^2
 \quad(v\in G_{p,u}),
\end{equation}
where $b>0$ is absolute. Put
\[
 W=\E|\ip{X}{u}|^{p-2}.
\]
The logarithmic Berwald inequality gives
\begin{equation}\label{eq:adjacent-moments}
 W\geq c m^{p-2};
\end{equation}
see Borell \cite{Borell}. Define a probability measure $\nu$ by
\[
 d\nu(x)=\frac{|\ip{x}{u}|^{p-2}}{W}\,d\mu_K(x).
\]
Then
\[
 Q_{p,u}(v)
 =
 \frac{W}{m^{p-2}}\int_K\ip{x}{v}^2\,d\nu(x)
 \geq
 c\int_K\ip{x}{v}^2\,d\nu(x).
\]

Let $F$ be the spectral subspace of $Q_{p,u}$ corresponding to
eigenvalues smaller than $b$, and let $r=\dim F$. If
$e_1,\dots,e_r$ is an orthonormal basis of $F$, then
\[
 \int_K|P_Fx|^2\,d\nu(x)
 \leq
 C\sum_{j=1}^r Q_{p,u}(e_j)
 \leq
 Cbr.
\]
Choose an absolute $\delta>0$, and then $b>0$ sufficiently small. By
Markov's inequality,
\[
 \nu\bigl(|P_Fx|\leq\delta\sqrt r\bigr)\geq\frac34.
\]
For
\[
 A=\{x\in K:|P_Fx|\leq\delta\sqrt r\},
\]
this and \eqref{eq:adjacent-moments} imply
\[
 \int_A|\ip{x}{u}|^{p-2}\,d\mu_K(x)
 \geq
 c m^{p-2}.
\]
Hölder's inequality gives the reverse estimate
\[
 \int_A|\ip{x}{u}|^{p-2}\,d\mu_K(x)
 \leq
 m^{p-2}\mu_K(A)^{2/p},
\]
and hence
\begin{equation}\label{eq:M-lower-small-ball}
 \mu_K(A)\geq e^{-Cp}.
\end{equation}
The random vector $P_FX$ is isotropic and log-concave in $F$. Applying
\eqref{eq:small-ball} at the origin gives
\[
 \mu_K(A)\leq(C\delta)^r\leq e^{-cr}
\]
when $\delta$ is sufficiently small. Comparing this with
\eqref{eq:M-lower-small-ball} yields $r\leq Cp$, proving
\eqref{eq:M-good-space}.

Set
\[
 E_{p,u}=G_{p,u}\cap\ker\ell_{p,u}.
\]
Then $\codim E_{p,u}\leq Cp$, and \eqref{eq:M-Hessian} gives
\[
 D^2\Phi_p(u)[v,v]\geq cp|v|^2,
 \qquad v\in E_{p,u}.
\]
Let $H(u)=\nabla^2\Phi(u)$. All the forms $D^2\Phi_q(u)$ are positive
semidefinite, so the min--max principle gives
\[
 \#\{i:\lambda_i(H(u))<1+cp\}\leq Cp.
\]
Choosing a dyadic $p$ comparable to a spectral threshold $t$, and
using $H(u)\succeq\Id_n$ and the trivial bound by $n$ for $t\gtrsim n$,
we obtain
\[
 N_u(t):=\#\{i:\lambda_i(H(u))<t\}\leq Ct,
 \qquad t\geq1.
\]
Consequently, by the layer-cake formula,
\[
 \Tr H(u)^{-1}
 =
 \int_0^1N_u(1/s)\,ds
 \leq
 n\int_0^{1/n}ds+C\int_{1/n}^1\frac{ds}{s}
 \leq
 C\log(n).
\]
This proves \eqref{eq:M-inverse-trace-goal}. By
\eqref{eq:spherical-M},
\[
 \int_{\Sph^{n-1}}\norm{\theta}_R^2\,d\sigma(\theta)
 \leq
 C\frac{\log(n)}{n}.
\]
Substitution in \eqref{eq:M-reduction} completes the proof.
\end{proof}

\section{Laplace bodies and tilted covariances}

From now on, $K=-K$, $\Cov(\mu_K)=\Id_n$, and we abbreviate
$\Lambda=\Lambda_K$. For $p\geq2$, set
\[
 \Lambda_p=\{\xi\in\R^n:\Lambda(\xi)\leq p\},
 \qquad
 L_p(K)=p\Lambda_p^\circ.
\]
The comparison theorem of Klartag and E.~Milman
\cite[Lemma~2.3] {KlartagMilmanLaplace} gives that for $2\leq p\leq n$
\begin{equation}\label{eq:Laplace-centroid}
 cZ_p(K)\subseteq L_p(K)\subseteq CZ_p(K).
\end{equation}
Together with the standard inclusions
$K\subseteq CZ_n(K)\subseteq C(n/p)Z_p(K)$, this yields
\begin{equation}\label{eq:K-in-Lp}
 K\subseteq C\frac npL_p(K).
\end{equation}

For $\xi\in\R^n$, let $\mu_\xi$ be the exponential tilt
\[
 d\mu_\xi(x)
 =
 e^{\ip{\xi}{x}-\Lambda(\xi)}\,d\mu_K(x).
\]
Then
\[
 \nabla\Lambda(\xi)=\E_{\mu_\xi}X,
 \qquad
 \nabla^2\Lambda(\xi)=\Cov(\mu_\xi).
\]

\begin{proposition}
\label{prop:tilted-covariance}
There is an absolute constant $D>0$ such that, for every
$\xi\in\R^n$ with $p:=\Lambda(\xi)\geq2$,
\[
 \#\bigl\{i:\lambda_i(\Cov(\mu_\xi))>D\bigr\}
 \leq
 2p.
\]
\end{proposition}

\begin{proof}
Let $A=\Cov(\mu_\xi)$ and let $F$ be the span of the eigenvectors of
$A$ with eigenvalue larger than $D$. Write $r=\dim F$, and let
\[
 \alpha=(P_F)_*\mu_K,
 \qquad
 \beta=(P_F)_*\mu_\xi.
\]
Then $\alpha$ is isotropic in $F$, whereas
$B:=\Cov(\beta)\succeq D\Id_F$. Set
\[
 E=\sqrt{2r}\,B_2^F.
\]
Since $\int_F|y|^2\,d\alpha(y)=r$, Markov's inequality gives
\begin{equation}\label{eq:alpha-E}
 \alpha(E)\geq\frac12.
\end{equation}

Let $b$ be the barycenter of $\beta$ and define
$T(y)=B^{-1/2}(y-b)$. The measure $\widetilde\beta=T_*\beta$ is
isotropic and log-concave on $F$. Moreover,
\[
 T(E)
 \subseteq
 -B^{-1/2}b+\sqrt{\frac{2r}{D}}\,B_2^F.
\]
By \eqref{eq:small-ball}, choosing $D$ sufficiently large gives
\begin{equation}\label{eq:beta-E}
 \beta(E)=\widetilde\beta(T(E))\leq e^{-2r}.
\end{equation}

Let
\[
    \widetilde E=P_F^{-1}(E).
\]
Then
\[
    a=\mu_K(\widetilde E)=\alpha(E),
    \qquad
    q=\mu_\xi(\widetilde E)=\beta(E).
\]
Since
\[
    \frac{d\mu_\xi}{d\mu_K}(x)
    =
    \exp\bigl(\langle \xi,x\rangle-\Lambda(\xi)\bigr)
    =
    \exp\bigl(\langle \xi,x\rangle-p\bigr),
\]
we have
\[
    \frac{q}{a}
    =
    \frac{1}{a}
    \int_{\widetilde E}
    \exp\bigl(\langle \xi,x\rangle-p\bigr)\,d\mu_K(x).
\]
By Jensen's inequality,
\begin{equation}\label{eq:Jensen-on-E}
    \log\frac{q}{a}
    \geq
    \frac{1}{a}
    \int_{\widetilde E}
    \bigl(\langle \xi,x\rangle-p\bigr)\,d\mu_K(x).
\end{equation}

Suppose first that \(a<1\). Applying the same argument to
\(\widetilde E^{\,c}\), we obtain
\begin{equation}\label{eq:Jensen-on-Ec}
    \log\frac{1-q}{1-a}
    \geq
    \frac{1}{1-a}
    \int_{\widetilde E^{\,c}}
    \bigl(\langle \xi,x\rangle-p\bigr)\,d\mu_K(x).
\end{equation}
Multiplying \eqref{eq:Jensen-on-E} by \(a\), multiplying
\eqref{eq:Jensen-on-Ec} by \(1-a\), and adding, we get
\[
\begin{aligned}
    a\log\frac{q}{a}
    +(1-a)\log\frac{1-q}{1-a}
    &\geq
    \int_K
    \bigl(\langle \xi,x\rangle-p\bigr)\,d\mu_K(x) \\
    &=
    -p,
\end{aligned}
\]
because \(\mu_K\) is centered. Therefore,
\[
    p
    \geq
    a\log\frac{a}{q}
    +(1-a)\log\frac{1-a}{1-q}.
\]
Since
\[
    a\log a+(1-a)\log(1-a)\geq-\log 2
\]
and
\[
    (1-a)\log\frac{1}{1-q}\geq0,
\]
it follows that
\[
    p
    \geq
    a\log\frac{1}{q}-\log 2.
\]
Using \eqref{eq:alpha-E} and \eqref{eq:beta-E}, namely
\[
    a\geq\frac12,
    \qquad
    q\leq e^{-2r},
\]
we conclude that
\[
    p
    \geq
    \frac12\log\frac{1}{q}-\log 2
    \geq
    r-\log 2.
\]

If \(a=1\), then \eqref{eq:Jensen-on-E} gives directly
\[
    \log q
    \geq
    \int_K
    \bigl(\langle \xi,x\rangle-p\bigr)\,d\mu_K(x)
    =
    -p,
\]
and hence
\[
    p\geq\log\frac{1}{q}\geq2r.
\]
Thus, in all cases,
\[
    r\leq p+\log 2\leq2p,
\]
where the last inequality follows from \(p\geq2\).
\end{proof}

\section{The mean-width estimate}

\begin{theorem}\label{thm:Mstar}
Let $K=-K\subset\R^n$, and assume that $\Cov(\mu_K)=\Id_n$. Then
\[
 M^*(K)\leq C\sqrt n\,\log(n).
\]
\end{theorem}

\begin{proof}
We may assume that $n$ is larger than an absolute constant. Fix a
small absolute $\eta>0$, and let
\[
 \mathcal P
 =
 \{2,4,8,\ldots,2^m\},
 \qquad
 2^m\leq\eta n<2^{m+1}.
\]
Put $N=1+|\mathcal P|$, so that $N\asymp\log(n)$.

For $p\in\mathcal P$, write
\[
 h_p=h_{L_p(K)},
 \qquad
 \varphi_p(u)=\frac12h_p(u)^2,
 \qquad
 \psi_p=\varphi_p^*=\frac12\norm{\cdot}_{L_p(K)}^2.
\]
We first record the curvature information supplied by
Proposition~\ref{prop:tilted-covariance}. Fix $u\neq0$, and let $s>0$
be determined by
\[
 \Lambda(su)=p.
\]
Set
\[
 \xi=su,
 \qquad
 m=\nabla\Lambda(\xi),
 \qquad
 A=\nabla^2\Lambda(\xi)=\Cov(\mu_\xi),
 \qquad
 \alpha=\ip{m}{u}.
\]
Since $L_p(K)=p\Lambda_p^\circ$,
\begin{equation}\label{eq:hp-p-over-s}
 h_p(u)=\frac ps.
\end{equation}
For $v\in m^\perp$, a direct differentiation of the level-set
identity
\[
 \Lambda\left(
 \frac{p}{h_p(u+tv)}(u+tv)
 \right)=p
\]
at $t=0$ gives
\begin{equation}\label{eq:Laplace-curvature}
 D^2\varphi_p(u)[v,v]
 =
 p\frac{h_p(u)}{\alpha}\ip{Av}{v}.
\end{equation}
Indeed, if $h(t)=h_p(u+tv)$, then $h'(0)=0$ for $v\perp m$,
$\xi'(0)=sv$, and the second derivative of the level-set identity gives
$h''(0)=p\ip{Av}{v}/\alpha$; since
$\varphi_p=h_p^2/2$, this is \eqref{eq:Laplace-curvature}. Finally, the
convexity of $r\mapsto\Lambda(ru)$ gives
\[
 p=\Lambda(su)\leq s\ip{\nabla\Lambda(su)}{u}=s\alpha,
\]
and hence, by \eqref{eq:hp-p-over-s},
\begin{equation}\label{eq:Laplace-curvature-upper}
 D^2\varphi_p(u)[v,v]
 \leq
 p\ip{Av}{v},
 \qquad v\in m^\perp.
\end{equation}

By Proposition~\ref{prop:tilted-covariance}, there is a subspace
$F\subseteq\R^n$ with $\codim F\leq2p$ such that
$A|_F\preceq D\Id_F$. Thus, on $F\cap m^\perp$,
\[
 D^2\varphi_p(u)\preceq Dp\Id.
\]
The min--max principle shows that
\begin{equation}\label{eq:varphi-high-spectrum}
 \#\{i:\lambda_i(\nabla^2\varphi_p(u))>C_1p\}
 \leq
 C_1p
\end{equation}
for an absolute $C_1$.

Legendre duality now gives the complementary lower spectral estimate
for $\psi_p$. Namely, for every $x\neq0$, there is a subspace
$E_{p,x}\subseteq\R^n$ such that
\begin{equation}\label{eq:psi-good-space}
 \codim E_{p,x}\leq C_1p,
 \qquad
 D^2\psi_p(x)|_{E_{p,x}}
 \succeq
 \frac1{C_1p}\Id_{E_{p,x}}.
\end{equation}

Define a symmetric convex body $S=S_K$ through its Minkowski
functional by
\begin{equation}\label{eq:S-def}
 \norm{x}_S^2
 =
 \frac1N
 \left[
 \frac{|x|^2}{n^2}
 +
 \sum_{p\in\mathcal P}
 \left(\frac pn\right)^2
 \norm{x}_{L_p(K)}^2
 \right].
\end{equation}
The right-hand side is the square of a norm. We first compare $S$ with
$K$. If $x\in K$, then \eqref{eq:K-in-Lp} gives
\[
 \frac pn\norm{x}_{L_p(K)}\leq C,
 \qquad p\in\mathcal P.
\]
The same standard inclusion with $p=2$, together with
$L_2(K)\simeq Z_2(K)=B_2^n$, gives $|x|/n\leq C$. Hence
$\norm{x}_S\leq C$, and therefore
\begin{equation}\label{eq:K-in-S}
 K\subseteq CS.
\end{equation}

Set
\[
 \psi(x)=\frac12\norm{x}_S^2
 =
 \frac1N
 \left[
 \frac{|x|^2}{2n^2}
 +
 \sum_{p\in\mathcal P}
 \left(\frac pn\right)^2\psi_p(x)
 \right]
\]
and
\[
 \varphi=\psi^*=\frac12h_S^2.
\]
Let $H(x)=\nabla^2\psi(x)$ and denote its eigenvalues in increasing
order by
\[
 0<\mu_1(x)\leq\cdots\leq\mu_n(x).
\]
From \eqref{eq:psi-good-space}, for each $p\in\mathcal P$ and
$v\in E_{p,x}$,
\[
 \ip{H(x)v}{v}
 \geq
 \frac1N\left(\frac pn\right)^2
 D^2\psi_p(x)[v,v]
 \geq
 \frac{p}{C_1Nn^2}|v|^2.
\]
Since $\codim E_{p,x}\leq C_1p$, the min--max principle yields
\begin{equation}\label{eq:H-counting-Mstar}
 \#\left\{i:\mu_i(x)<\frac{p}{C_1Nn^2}\right\}
 \leq
 C_1p.
\end{equation}
The Euclidean term in \eqref{eq:S-def} gives
\[
 H(x)\succeq\frac1{Nn^2}\Id_n.
\]
We now pass from \eqref{eq:H-counting-Mstar} to an ordered-eigenvalue
estimate. Let $p_{\max}=2^m$. For bounded $i$, the Euclidean term is
enough. If $i$ is larger than an absolute constant and
$i\leq2C_1p_{\max}$, choose a dyadic $p\in\mathcal P$ such that
\[
 \frac{i}{4C_1}\leq p<\frac{i}{2C_1}.
\]
Then $C_1p<i$, and \eqref{eq:H-counting-Mstar} implies
\[
 \mu_i(x)
 \geq
 \frac{p}{C_1Nn^2}
 \geq
 c\frac{i}{Nn^2}.
\]
Finally, applying \eqref{eq:H-counting-Mstar} with $p=p_{\max}$
shows that
\[
 \mu_{\lfloor C_1p_{\max}\rfloor+1}(x)
 \geq
 \frac{p_{\max}}{C_1Nn^2}.
\]
Since $p_{\max}\simeq n$ and the eigenvalues are increasing, the
same lower bound, and hence a lower bound by $ci/(Nn^2)$, holds for all
remaining indices. Choosing $\eta>0$ sufficiently small only ensures
that the indicated index is at most $n$. We have therefore proved
\begin{equation}\label{eq:linear-growth-Mstar}
 \mu_i(x)
 \geq
 c\frac{i}{Nn^2},
 \qquad
 i=1,\ldots,n.
\end{equation}

It follows that
\begin{equation}\label{eq:inverse-trace-Mstar}
 \Tr H(x)^{-1}
 \leq
 CNn^2\sum_{i=1}^n\frac1i
 \leq
 Cn^2(\log(n))^2.
\end{equation}
By Legendre duality,
\[
 \nabla^2\varphi(u)
 =
 \left[\nabla^2\psi\bigl(\nabla\varphi(u)\bigr)\right]^{-1},
\]
and hence \eqref{eq:inverse-trace-Mstar} gives
\[
 \Delta\varphi(u)\leq Cn^2(\log(n))^2,
 \qquad u\neq0.
\]
Since $\varphi=h_S^2/2$ is $2$-homogeneous, the spherical Laplacian
identity gives
\[
 n\int_{\Sph^{n-1}}h_S(\theta)^2\,d\sigma(\theta)
 =
 \int_{\Sph^{n-1}}\Delta\varphi(\theta)\,d\sigma(\theta).
\]
Therefore
\[
 \int_{\Sph^{n-1}}h_S(\theta)^2\,d\sigma(\theta)
 \leq
 Cn(\log(n))^2,
\]
and Cauchy--Schwarz yields
\[
 M^*(S)\leq C\sqrt n\,\log(n).
\]
Finally, \eqref{eq:K-in-S} implies $h_K\leq Ch_S$, and hence
\[
 M^*(K)\leq CM^*(S)\leq C\sqrt n\,\log(n).
\]
\end{proof}

\begin{remark}[The traditional volume-one normalization]
Suppose that $|K|=1$ and
$\Cov(\mu_K)=L_K^2\Id_n$. Applying Theorems~\ref{thm:M} and
\ref{thm:Mstar} to $L_K^{-1}K$ gives
\[
 L_KM(K)\leq C\frac{\log(n)}{\sqrt n}
\]
and
\[
 M^*(K)\leq CL_K\sqrt n\,\log(n).
\]
Thus the second estimate improves the earlier
$CL_K\sqrt n(\log(n))^2$ bound of E.~Milman by one logarithmic factor,
under the slicing input used here.
\end{remark}

\begin{remark}[What the arguments give without slicing]
\label{rem:without-slicing}
We explain more precisely the dependence on the slicing input. Assume
first that \(K\) is in probabilistic isotropic position, and let
\(\mu_K\) be the uniform probability measure on \(K\).

In the proof of the estimate for \(M(K)\), the slicing theorem is used
only for the marginals \(P_F\mu_K\) appearing in the
weighted-covariance argument. Define
\[
    \mathcal L_M(K)
    =
    \sup_{F\subseteq\R^n}
    L\bigl((P_F)_*\mu_K\bigr).
\]
Since \(P_F\mu_K\) is isotropic on \(F\), its density satisfies
\[
    \Pp\left(
        |P_FX|\leq\delta\sqrt r
    \right)
    \leq
    \left(
        C\delta\,
        L\bigl((P_F)_*\mu_K\bigr)
    \right)^r,
    \qquad r=\dim F.
\]
Consequently, in the weighted-covariance lemma one must choose
\[
    \delta
    \simeq
    \frac{1}{\mathcal L_M(K)}.
\]
The lower bound for the relevant weighted covariance is then of order
\(\mathcal L_M(K)^{-2}\), and hence the curvature supplied by the
\(p\)-th summand is of order
\[
    \frac{p}{\mathcal L_M(K)^2}
\]
outside a subspace of codimension \(Cp\). The spectral-counting and
inverse-trace arguments therefore give
\[
    \Tr\bigl[\nabla^2\Phi(u)\bigr]^{-1}
    \leq
    C\mathcal L_M(K)^2\log(n).
\]
After the spherical-Laplacian argument and the comparison between
\(R_K\) and \(K\), this yields
\[
    M(K)
    \leq
    C\mathcal L_M(K)
    \frac{\log(n)}{\sqrt n}.
    \tag{9.1}
\]

The corresponding parameter in the \(M^*(K)\)-argument involves
exponential tilts. For \(\xi\in\R^n\), write
\[
    m_\xi=\E_{\mu_\xi}X,
    \qquad
    A_\xi=\Cov(\mu_\xi),
\]
and let
\[
    \widehat\mu_\xi
    =
    \bigl(A_\xi^{-1/2}\bigr)_*
    \bigl(\mu_\xi(\,\cdot+m_\xi)\bigr)
\]
be the isotropic affine image of \(\mu_\xi\). Define
\[
    \mathcal L_{M^*}(K)
    =
    \sup_{\substack{
        \xi\in\R^n,\ \Lambda(\xi)\leq\eta n\\
        F\subseteq\R^n
    }}
    L\bigl((P_F)_*\widehat\mu_\xi\bigr),
\]
where \(\eta>0\) is the fixed absolute constant used in the dyadic
construction.

Suppose that \(F\) is an \(r\)-dimensional spectral subspace on which
\[
    \Cov(\mu_\xi)|_F
    \succeq
    D\,\Id_F.
\]
After whitening the projected tilted measure, the general small-ball
estimate gives
\[
    (P_F)_*\mu_\xi
    \left(
        \sqrt{2r}\,B_2^F
    \right)
    \leq
    \left(
        \frac{C\mathcal L_{M^*}(K)}{\sqrt D}
    \right)^r.
\]
Choosing
\[
    D
    =
    C\mathcal L_{M^*}(K)^2
\]
and repeating the direct comparison between the original measure and
its exponential tilt, one obtains
\[
    \#\left\{
        i:
        \lambda_i\bigl(\Cov(\mu_\xi)\bigr)
        >
        C\mathcal L_{M^*}(K)^2
    \right\}
    \leq
    2\Lambda(\xi).
    \tag{9.2}
\]
The curvature formula for the Laplace body then gives, outside a
subspace of codimension \(Cp\),
\[
    \nabla^2
    \left(
        \frac12 h_{L_p(K)}^2
    \right)
    \preceq
    Cp\,\mathcal L_{M^*}(K)^2\Id.
\]
After Legendre duality, the corresponding summand in the Hessian of
the aggregate gauge has curvature bounded below by
\[
    \frac{c}{p\,\mathcal L_{M^*}(K)^2}.
\]
Thus the linear eigenvalue growth and the inverse-trace estimate lose
a factor \(\mathcal L_{M^*}(K)^2\). Taking the square root in the final
spherical \(L_2\)-estimate yields
\[
    M^*(K)
    \leq
    C\mathcal L_{M^*}(K)
    \sqrt n\,\log(n).
    \tag{9.3}
\]

Since
\[
    \mathcal L_M(K)\leq L_n,
    \qquad
    \mathcal L_{M^*}(K)\leq L_n,
\]
equations \((9.1)\) and \((9.3)\) imply
\[
    M(K)
    \leq
    C L_n\frac{\log(n)}{\sqrt n},
    \qquad
    M^*(K)
    \leq
    C L_n\sqrt n\,\log(n).
\]

Finally, let \(K\) have volume one and satisfy
\[
    \Cov(\mu_K)=L_K^2\Id_n.
\]
For \(\widetilde K=L_K^{-1}K\), one has
\[
    M(\widetilde K)=L_KM(K),
    \qquad
    M^*(\widetilde K)=\frac{1}{L_K}M^*(K).
\]
Applying the preceding estimates to \(\widetilde K\) gives
\[
    M(K)
    \leq
    C\frac{L_n}{L_K}
    \frac{\log(n)}{\sqrt n},
\]
whereas
\[
    M^*(K)
    \leq
    C L_KL_n\sqrt n\,\log(n).
\]
The different positions of \(L_K\) in these two formulas are entirely
due to homogeneity: the mean gauge has degree \(-1\), while the mean
width has degree \(1\). In particular, these estimates do not provide
an independent proof of slicing; the slicing theorem is precisely what
turns the hereditary quantities
\(\mathcal L_M(K)\) and \(\mathcal L_{M^*}(K)\) into absolute
constants.
\end{remark}

\begin{remark}[The non-symmetric case]
\label{rem:non-symmetric}
The symmetry assumption is not essential for either estimate, although
the two arguments require slightly different modifications. Let
\(K\subset\R^n\) be centered, not necessarily symmetric, and assume
that
\[
    \Cov(\mu_K)=\Id_n.
\]

For the mean-gauge estimate, one replaces the symmetric centroid
bodies by the one-sided centroid bodies
\[
    h_{Z_p^+(K)}(u)
    =
    \left(
        \int_K \ip{x}{u}_+^p\,d\mu_K(x)
    \right)^{1/p},
    \qquad p\geq1,
\]
and defines \(R=R_K\) by
\[
    h_R(u)^2
    =
    |u|^2
    +
    \sum_{p\in\mathcal D_n}
    h_{Z_p^+(K)}(u)^2.
\]
This is the support function of a convex body, although \(R\) need
not be symmetric. Indeed, each \(h_{Z_p^+(K)}\) is non-negative,
positively homogeneous and subadditive, and hence their quadratic
aggregate is again subadditive.

For \(x\in K\),
\[
    \ip{x}{u}_+
    \leq
    h_K(u),
\]
and therefore
\[
    h_{Z_p^+(K)}(u)\leq h_K(u).
\]
Moreover, if \(Y=\ip{X}{u}\), then centeredness gives
\[
    \E Y_+
    =
    \frac12\E|Y|,
\]
whereas the one-dimensional reverse moment inequality yields
\[
    \E|Y|
    \geq
    c\bigl(\E Y^2\bigr)^{1/2}
    =
    c|u|.
\]
Consequently,
\[
    |u|\leq C h_K(u),
\]
and hence, exactly as in the symmetric case,
\[
    R\subseteq C\sqrt{\log(n)}\,K.
\]

The Hessian computation is unchanged, with
\(|\ip{X}{u}|^{p-2}\) replaced by \(\ip{X}{u}_+^{p-2}\). The adjacent
moment estimate needed in the weighted-covariance lemma becomes
\[
    \E\ip{X}{u}_+^{p-2}
    \geq
    c
    \left(
        \E\ip{X}{u}_+^p
    \right)^{(p-2)/p}.
\]
To see this, one conditions the one-dimensional log-concave random
variable \(\ip{X}{u}\) on the positive half-line, applies the
logarithmic Berwald inequality to the resulting log-concave density,
and uses Gr\"unbaum's inequality to bound the mass of the positive
half-line from below. The weighted small-ball argument, the spectral
counting estimate and the Legendre-duality step then proceed without
any further change. Thus
\[
    M(K)
    \leq
    C\frac{\log(n)}{\sqrt n}
\]
also holds without symmetry.

For the mean-width estimate, no one-sided modification has to be
introduced separately. Define, as before,
\[
    \Lambda(\xi)
    =
    \log\int_K e^{\ip{\xi}{x}}\,d\mu_K(x),
    \qquad
    \Lambda_p
    =
    \{\xi\in\R^n:\Lambda(\xi)\leq p\},
\]
and
\[
    L_p(K)=p\Lambda_p^\circ.
\]
The functions \(\Lambda\) and the bodies \(L_p(K)\) need not be even.
Nevertheless, the non-symmetric form of the Laplace--centroid
comparison of Klartag and E.~Milman
\cite[Lemma~2.3]{KlartagMilmanLaplace} gives
\[
    c Z_p(K)
    \subseteq
    L_p(K)
    \subseteq
    C Z_p(K),
    \qquad p\geq1,
\]
where \(Z_p(K)\) is defined using the absolute moments
\[
    h_{Z_p(K)}(u)
    =
    \left(
        \int_K |\ip{x}{u}|^p\,d\mu_K(x)
    \right)^{1/p}.
\]
In particular, using
\[
    K
    \subseteq
    \text{conv}(K,-K)
    \subseteq
    C Z_n(K)
    \subseteq
    C\frac{n}{p}Z_p(K),
\]
we obtain
\[
    K\subseteq C\frac{n}{p}L_p(K),
    \qquad 2\leq p\leq n.
\]

The auxiliary body \(S=S_K\) is consequently defined by the same
formula
\[
    \norm{x}_S^2
    =
    \frac1N
    \left[
        \frac{|x|^2}{n^2}
        +
        \sum_{p\in\mathcal P}
        \left(\frac pn\right)^2
        \norm{x}_{L_p(K)}^2
    \right].
\]
Although the gauges appearing here need not be even, the right-hand
side is still the square of a gauge: the triangle inequality follows
by applying the Euclidean triangle inequality to the vector of the
individual non-negative gauges. The containment argument therefore
still gives
\[
    K\subseteq CS.
\]

The tilted-covariance estimate and the curvature formula for
\[
    \frac12h_{L_p(K)}^2
\]
use only the centeredness and log-concavity of \(\mu_K\), and not its
symmetry. Legendre duality is equally valid for non-symmetric gauges:
for every convex body \(A\) containing the origin in its interior,
\[
    \left(\frac12\norm{\cdot}_A^2\right)^*
    =
    \frac12h_A^2.
\]
Thus the lower spectral estimate for the Hessian of
\(\norm{\cdot}_S^2/2\), the inverse-trace bound and the spherical
Laplacian argument remain unchanged. It follows that
\[
    M^*(K)
    \leq
    C\sqrt n\,\log(n)
\]
for every centered convex body in probabilistic isotropic position,
without any symmetry assumption.

The versions without the slicing input stated in
Remark~\ref{rem:without-slicing} extend in the same way, with the
isotropic constants of the relevant marginals and whitened tilted
marginals retained in the corresponding estimates.
\end{remark}

\end{document}